\documentclass[12pt]{amsart}
\usepackage{amsfonts,amsmath,amscd,amssymb,paralist,amsthm}
\usepackage{mathrsfs}
\usepackage{amsxtra}
\usepackage[mathscr]{eucal}
\usepackage{graphics}
\usepackage{latexsym,graphicx,verbatim,nameref}
\usepackage[numbers,sort&compress]{natbib}
\usepackage[all,cmtip]{xy}
\usepackage{hyperref}
\usepackage{fancyhdr}
\usepackage{color}
\usepackage{tikz}
\usepackage{tikz-cd}
\usetikzlibrary{arrows,scopes,matrix}
\numberwithin{equation}{section}

\theoremstyle{definition}

\makeatletter
\newtheorem*{rep@theorem}{\rep@title}
\newcommand{\newreptheorem}[2]{%
\newenvironment{rep#1}[1]{%
 \def\rep@title{#2 \ref{##1}}%
 \begin{rep@theorem}}%
 {\end{rep@theorem}}}
\makeatother

\newreptheorem{theorem}{Theorem}
\newreptheorem{corollary}{Corollary}
\newreptheorem{lemma}{Lemma}

\newtheorem{theorem}{Theorem}[section]
\newtheorem{corollary}[theorem]{Corollary}
\newtheorem{lemma}[theorem]{Lemma}
\newtheorem{proposition}[theorem]{Proposition}
\newtheorem*{theorem*}{Theorem}
\newtheorem*{proposition*}{Proposition}

\newtheorem{definition}[theorem]{Definition}

\newtheorem*{claim*}{Claim}

\newtheorem*{conjecture*}{Conjecture}
\newtheorem*{observation*}{Observation}
\newtheorem*{question*}{Question}

\newcommand{\fai}{\varphi}

\newcommand{\Hinv}{H_{\rm inv}}
\newcommand{\Tr}{{\rm Tr}}
\def\A{G}

\begin{document}
\title{Mean ergodic theorem for amenable discrete quantum groups and a Wiener type theorem for compact metrizable groups}
\author{Huichi Huang}
\address{Huichi Huang, Mathematisches Institut, Universit\"{a}t M\"{u}nster, Einsteinstr. 62,  M\"{u}nster, 48149, Germany}
\email{huanghuichi@uni-muenster.de}
\date{\today}
\thanks{{\it Supported by:}  ERC Advanced Grant No. 267079}
\begin{abstract}
We prove a mean ergodic theorem for amenable discrete quantum groups. As an application, we prove  a Wiener type theorem for continuous measures on compact metrizable groups.
\end{abstract}

\maketitle
\tableofcontents

\section{Introduction}\

A countable discrete group $\Gamma$ is  called {\bf amenable} if there exists a sequence $\{F_n\}_{n=1}^\infty$ (called a right F{\o}lner sequence) consisting of finite subsets $F_n$ of $\Gamma$ such that
$$\lim_{n\to\infty}\frac{1}{|F_n|}|F_ns\Delta F_n|=0$$ for every $s\in\Gamma$.

Let $(X,\mathcal{B},\mu,\Gamma)$ be a dynamical system consisting of  a countable discrete amenable group $\Gamma$ with a measure-preserving action on a probability space $(X,\mathcal{B},\mu)$.

Recall that von Neumann's mean ergodic theorem for amenable group actions on measure spaces says the following.

\medskip

\begin{theorem}~[Measure space version of von Neumann's mean ergodic theorem]~\cite[3.33]{Glasner2003}\

Let  $\{F_n\}_{n=1}^\infty$ be a right F{\o}lner sequence of \,$\Gamma$. Then for every $f\in L^2(X,\mu)$, the sequence
$\frac{1}{|F_n|}\sum_{s\in F_n} s\cdot f$ converges to $Pf$ with respect to $L^2$ norm, where $P$ is the orthogonal projection from $L^2(X,\mu)$ onto the space $\{g\in L^2(X,\mu)|s\cdot g=g \,{\rm for \,all}\, s\in\Gamma\}.$
\end{theorem}

R. Duvenhage proves a generalization of von Neumann's mean ergodic theorem for coactions of amenable quantum groups on von Neumann algebras~(noncommutative measure spaces)~\cite[Theorem 3.1.]{Duvenhage2008}. Later a more general version is proved by V. Runge and A. Viselter~\cite[Theorem 2.2]{RV2014}.

There is also a version of von Neumann's mean ergodic theorem for amenable group actions on Hilbert spaces, which says the following.

\medskip

\begin{theorem}~[Hilbert space version of von Neumann's mean ergodic theorem]\

Let  $\{F_n\}_{n=1}^\infty$ be a right F{\o}lner sequence of a countable discrete amenable group $\Gamma$ and $\pi:\Gamma\to B(H)$ be a unitary representation of\,  $\Gamma$ on a Hilbert space $H$. Then
$$\lim_{n\to\infty}\frac{1}{|F_n|}\sum_{s\in F_n} \pi(s)=P$$ under strong operator topology on $B(H)$, where $P$ is  the orthogonal projection from $H$ onto $H_\Gamma=\{x\in H\,|\,\pi(s)x=x \, {\rm for\, all}\, s\in\Gamma\}.$
\end{theorem}

The group $C^*$-algebra $C^*(\Gamma)$  equals  $C(G)$ for a coamenable compact quantum group $G$ with the dual group $\widehat{G}=\Gamma$. The counit $\varepsilon$ of $G$ is given by $\varepsilon(\delta_s)=1$ for all $s\in\Gamma$. Hence $H_\Gamma=\{x\in H\,|\,\pi(a)x=\varepsilon(a)x \, {\rm for\, all}\, a\in C^*(\Gamma)\}$. With these in mind, the Hilbert space version of von Neumann's mean ergodic theorem could be reformulated in the framework of compact quantum groups as follows.

Suppose  $G$ is a coamenable compact quantum group such that the dual $\widehat{G}$ is  a countable discrete amenable group $\Gamma$. Let $\{F_n\}_{n=1}^\infty$ be a right F{\o}lner sequence of $\Gamma$ and $\pi:C(G)=C^*(\Gamma)\to B(H)$ be a  representation of $C^*(\Gamma)$ on a Hilbert space $H$. Then
$$\lim_{n\to\infty}\frac{1}{|F_n|}\sum_{s\in F_n} \pi(s)=P$$ under strong operator topology on $B(H)$, where $P$ is  the orthogonal projection from $H$ onto $H_\Gamma=\{x\in H\,|\,\pi(a)x=\varepsilon(a)x \, {\rm for\, all}\, a\in C^*(\Gamma)\}.$

D. Kyed proves that a compact quantum group $G$ is coamenable iff there exists a right F{\o}lner sequence $\{F_n\}_{n=1}^\infty$ of finite subsets in its dual $\widehat{G}$, that is to say, $G$ is a coamenable compact quantum group iff $\widehat{G}$ is an amenable discrete quantum group~\cite[Definition 4.9.]{Kyed2008}.~\footnote{Existence of F{\o}lner sequence for  Kac type compact quantum groups is shown by Z. Ruan in~\cite{Ruan1996}. Also cf.~\cite{Tomatsu2006}.} So it is natural to ask for a generalization of the Hilbert space version of von Neumann's mean ergodic theorem to all amenable discrete quantum groups. This is the main result of the paper.

\medskip

\begin{reptheorem}{MeanErgodic}~[Mean ergodic theorem for amenable discrete quantum groups]\

Let $G$ be a coamenable compact quantum group  with counit $\varepsilon$ and  $\{F_n\}_{n=1}^\infty$ be a right F{\o}lner sequence of $\widehat{G}$. For a representation $\pi:A=C(G)\to B(H)$, we have
\begin{equation}~\label{EqCon}
\lim_{n\to\infty}\frac{1}{|F_n|_w}\sum_{\alpha\in F_n}d_\alpha\pi(\chi(\alpha))=P
\end{equation}
under strong operator topology, where $P$ is the orthogonal projection from $H$ onto $\Hinv=\{x\in H\,|\,\pi(a)x=\varepsilon(a) x\, {\rm for\, all}\, a\in A\}$.
\end{reptheorem}

Here  $|F_n|_w$ stands for the weighted cardinality of $F_n$. Definitions of $|F_n|_w$, $d_\alpha$ and $\chi(\alpha)$ are in section 2.

The left hand side of Equation (1.1) involves both representation of  a coamenable compact quantum group $G$ and that of its discrete quantum group dual $\widehat{G}$, so it illustrates some interactions between them.

The rest part of the paper aims at an application of Theorem~\ref{MeanErgodic}. Namely, we prove a Wiener type theorem for finite Borel measures on compact metrizable groups.

A finite Borel  measure $\mu$ on a compact metrizable space $X$ is called {\bf continuous}, or {\bf non-atomic} if $\mu\{x\}=0$ for every $x\in X$.

The following  theorem of N. Wiener expresses finite Borel measures on the unit circle via their Fourier coefficients~\cite{Wiener1933}.

\medskip

\begin{theorem}~[Wiener's Theorem]~\cite[Chapter 1, 7.13]{Katznelson2004}~\label{Wiener}\

For a finite Borel  measure $\mu$ on the unit circle $\mathbb{T}$ and every $z\in\mathbb{T}$, one have
$$\lim_{N\to\infty} \frac{1}{2N+1}\sum_{n=-N}^N\hat{\mu}(n)z^{-n}=\mu\{z\},$$ and
$$\lim_{N\to\infty} \frac{1}{2N+1}\sum_{n=-N}^N|\hat{\mu}(n)|^2=\sum_{x\in \mathbb{T}} \mu\{x\}^2.$$ Hence
$\mu$ is continuous iff
$$\lim_{N\to\infty} \frac{1}{2N+1}\sum_{n=-N}^N|\hat{\mu}(n)|^2=0.$$ Where $\hat{\mu}(n):=\int_{\mathbb{T}}z^n\,d\mu(z)$ for every $n\in \mathbb{Z}$ are Fourier coefficients of $\mu$.
\end{theorem}

There are various generalized Wiener's Theorems (we call such generalizations Wiener type theorems), say, a version for compact manifolds~\cite[Chapter XII, Theorem 5.1]{Taylor1981}, a version for compact Lie groups by M. Anoussis and A. Bisbas~\cite[Theorem 7]{AnoussisBisbas2000}, and a version  for compact homogeneous manifolds by M. Bj\"{o}rklund and A. Fish~\cite[Lemma 2.1]{BjorklundFish2009}.

We apply the above mean ergodic theorem~(Theorem~\ref{MeanErgodic}) to get a version of Wiener type theorem  on compact metrizable groups. This version differs from previous ones mainly in two aspects: firstly we don't require smoothness on  spaces, secondly we use a different F{\o}lner  condition.

\medskip

\begin{reptheorem}{WienerType}~[Wiener type theorem for compact metrizable groups]\

Let $G$ be a compact metrizable group. Given a $y$ in $G$ and a right F{\o}lner sequence $\{F_n\}_{n=1}^\infty$ of $\widehat{G}$, for a finite Borel measure $\mu$ on $G$, one have
$$\lim_{n\to\infty}\frac{1}{|F_n|_w}\sum_{\alpha\in F_n}d_\alpha\sum_{1\leq i,j\leq d_\alpha}\mu(u^\alpha_{ij})\overline{u^\alpha_{ij}(y)}=\mu\{y\},$$
and $$\lim_{n\to\infty}\frac{1}{|F_n|_w}\sum_{\alpha\in F_n}d_\alpha\sum_{1\leq i,j\leq d_\alpha}|\mu(u^\alpha_{ij})|^2=\sum_{x\in G} \mu\{x\}^2.$$
Hence $\mu$ is continuous iff
$$\lim_{n\to\infty}\frac{1}{|F_n|_w}\sum_{\alpha\in F_n}d_\alpha\sum_{1\leq i,j\leq d_\alpha}|\mu(u^\alpha_{ij})|^2=0.$$
\end{reptheorem}

Here $u^\alpha_{ij}$'s are the matrix coefficients of the irreducible unitary representation $\alpha$ of $G$. Cf. section 2 for the precise definition.

The paper is organized as follows.

In the preliminary section, we collect some basic facts in compact quantum group theory. In section 3, we prove mean ergodic theorem, i.e., Theorem~\ref{MeanErgodic}. As a consequence, we obtain Corollary~\ref{PureStates}, which is used in section 4 to prove Theorem~\ref{WienerType}.

\section*{Acknowledgements}

I thank Martijn Caspers for pointing out reference~\cite{Kyed2008} to me which motivates the article. I am  grateful to Hanfeng Li and Shuzhou Wang for their comments. I  thank Ami Viselter for reminding me of some preceding works. At last but not least I  thank the anonymous referee and  the editor for their comments and suggestions, which greatly improve the readability of the article. 

\section{Preliminary}\

\subsection{Conventions}\

Within this paper, we use $B(H,K)$ to denote the space of bounded linear operators from a Hilbert space $H$ to another Hilbert space $K$, and $B(H)$ stands for $B(H,H)$.

A net  $\{T_\lambda\}\subset B(H)$ converges to $T\in B(H)$ under strong operator topology~(SOT) if $T_\lambda x\to Tx$ for every $x\in H$, and  $\{T_\lambda\}$  converges to $T\in B(H)$ under weak operator topology~(WOT) if
$\langle T_\lambda x,y\rangle\to \langle Tx,y\rangle$ for all $x,y\in H$.

The notation $A \otimes B$ always means the minimal tensor product of two $C^*$-algebras $A$ and $B$.

For a state $\fai$ on a unital $C^*$-algebra $A$, we use $L^2(A,\fai)$ to denote the Hilbert space of GNS representation of $A$ with respect to $\fai$. The image of  an $a\in A$ in $L^2(A,\fai)$ is denoted by $\hat{a}$.

In this paper all $C^*$-algebras are assumed to be unital and separable.

\subsection{Some  facts about compact quantum groups}\

Compact quantum groups are  noncommutative analogues of compact groups. They are introduced by S. L. Woronowicz~\cite{Woronowicz1987,Woronowicz1998}.

\medskip

\begin{definition}
A {\bf compact quantum group} is a pair $(A,\Delta)$ consisting of a unital $C^*$-algebra $A$ and a unital $*$-homomorphism $$\Delta: A\rightarrow A\otimes A$$ such that
\begin{enumerate}
\item $(id\otimes\Delta)\Delta=(\Delta\otimes id)\Delta$.
\item $\Delta(A)(1\otimes A)$ and $\Delta(A)(A\otimes 1)$ are dense in $A\otimes A$.
\end{enumerate}
The $*$-homomorphism $\Delta$ is called the {\bf coproduct} of $\A$.
\end{definition}

One may think of $A$ as $C(G)$, the $C^*$-algebra of continuous functions on a compact quantum space $G$ with a quantum group structure. In the rest of the paper we write  a compact quantum group $(A,\Delta)$ as $G$.

There exists a unique state $h$ on $A$ such that $$(h\otimes id)\Delta(a)=(id\otimes h)\Delta(a)=h(a)1_A$$
 for all $a$ in $A$. The state $h$ is called the {\bf Haar measure} of $G$. Throughout this paper, we use $h$ to denote it.

For a compact quantum group $G$, there is a unique dense unital $*$-subalgebra $\mathcal{A}$ of $A$ such that
\begin{enumerate}
\item  $\Delta$  maps from $\mathcal{A}$ to $\mathcal{A}\odot \mathcal{A}$~(algebraic tensor product).
\item There exists a unique multiplicative linear functional $\varepsilon : \mathcal{A} \to \mathbb{C}$ and a linear map $\kappa : \mathcal{A} \to \mathcal{A}$ such that
$(\varepsilon \otimes id)\Delta(a) = (id \otimes \varepsilon) \Delta(a) = a$ and $m(\kappa \otimes id) \Delta(a) = m(id \otimes \kappa)\Delta(a) = \varepsilon(a)1$
for all $a \in \mathcal{A}$, where $m:\mathcal{A} \odot \mathcal{A}\to \mathcal{A}$ is the multiplication map. The functional $\varepsilon$ is called {\bf counit} and $\kappa$ the {\bf coinverse} of $C(G)$.
\end{enumerate}

Note that $\varepsilon$ is only densely defined and not necessarily bounded. If $\varepsilon$ is bounded and $h$ is faithful~($h(a^*a)=0$ implies $a=0$), then $G$ is called {\bf coamenable}~\cite{BMT2001}. Examples of coamenable compact quantum groups include $C(G)$ for a compact group $G$ and $C^*(\Gamma)$ for a  discrete amenable group \,$\Gamma$.

A nondegenerate (unitary) {\bf representation} $U$ of  a compact quantum group $G$ is an invertible~(unitary) element in $M(K(H)\otimes A)$ for some Hilbert space $H$ satisfying that $U_{12}U_{13}=(id\otimes \Delta)U$. Here $K(H)$ is the $C^*$-algebra of compact operators on $H$ and  $M(K(H)\otimes A)$ is the multiplier $C^*$-algebra of  $K(H)\otimes A$.

We  write $U_{12}$ and $U_{13}$ respectively for the images of  $U$ by two maps from $M(K(H)\otimes A)$ to $M(K(H)\otimes A\otimes A)$ where the first one is obtained by extending the map $x \mapsto x \otimes 1$ from $K(H) \otimes A$ to $K(H) \otimes A\otimes A$, and the second one is obtained by composing this map with the flip on the  last two factors. The Hilbert space $H$ is called the \textbf{carrier Hilbert space} of $U$. From now on, we always assume representations are nondegenerate. If the carrier Hilbert space $H$ is of finite dimension, then $U$ is called a finite dimensional representation of $G$.

For two representations $U_1$ and $U_2$ with the carrier Hilbert spaces $H_1$ and $H_2$ respectively, the set of
{\bf intertwiners}  between $U_1$ and $U_2$, ${\rm Mor}(U_1,U_2)$, is defined by
$${\rm Mor}(U_1,U_2)=\{T\in B(H_1,H_2)|(T\otimes 1)U_1=U_2(T\otimes 1)\}.$$
Two representations $U_1$ and $U_2$ are equivalent if there exists a bijection $T$ in ${\rm Mor}(U_1,U_2)$.
A representation $U$ is called {\bf irreducible} if ${\rm Mor}(U,U)\cong\mathbb{C}$.

Moreover, we have the following well-established facts about representations of compact quantum groups:
\begin{enumerate}
\item Every finite dimensional representation is equivalent to a unitary representation.
\item Every irreducible representation is  finite dimensional.
\end{enumerate}
Let $\widehat{G}$ be the set of equivalence classes of irreducible unitary representations of $G$. For every $\gamma\in \widehat{G}$, let $U^{\gamma}\in \gamma$  be unitary and $H_{\gamma}$ be its carrier Hilbert space with dimension $d_{\gamma}$. After fixing an orthonormal basis of $H_{\gamma}$, we can write $U^{\gamma}$ as  $(u^{\gamma}_{ij})_{1\leq i,j\leq d_{\gamma}}$ with $u^{\gamma}_{ij}\in A$, and
$$\Delta(u^{\gamma}_{ij})=\sum_{k=1}^{d_\gamma}u^{\gamma}_{ik}\otimes u^{\gamma}_{kj}$$ for all $1\leq i,j\leq d_\gamma$.

The matrix $\overline{U^{\gamma}}$ is still an irreducible representation~(not necessarily unitary) with the carrier Hilbert space $\bar{H}_\gamma$. It is called the {\bf conjugate} representation of $U^\gamma$ and the equivalence class of $\overline{U^{\gamma}}$ is denoted by $\bar{\gamma}$.

Given two finite dimensional representations  $\alpha,\beta$ of $G$,  fix orthonormal basises for $\alpha$ and $\beta$ and write $\alpha,\beta$ as $U^\alpha, U^\beta$ in matrix forms respectively. Define the {\bf direct sum}, denoted by $\alpha+\beta$ as an equivalence class of unitary representations of dimension $d_\alpha+d_\beta$ given by
$\bigl(\begin{smallmatrix}
U^\alpha&0\\ 0&U^\beta
\end{smallmatrix} \bigr)$, and
the {\bf tensor product}, denoted by $\alpha\beta$,  is an equivalence class of unitary representations of dimension $d_\alpha d_\beta$ whose matrix form is given by $U^{\alpha\beta}=U^\alpha_{13}U^\beta_{23}$.

The {\bf character} $\chi(\alpha)$ of a finite dimensional representation $\alpha$ is given by
$$\chi(\alpha)=\sum_{i=1}^{d_\alpha} u^\alpha_{ii}.$$ Note that $\chi(\alpha)$ is independent of choices of representatives of $\alpha$. Also $\|\chi(\alpha)\|\leq d_\alpha$ since $\sum_{k=1}^{d_\alpha} u^\alpha_{ik}(u^\alpha_{ik})^*=1$ for every $1\leq i\leq d_\alpha$. Moreover
$$\chi(\alpha+\beta)=\chi(\alpha)+\chi(\beta),\,\chi(\alpha\beta)=\chi(\alpha)\chi(\beta)\,{\rm and}\,\chi(\alpha)^*=\chi(\bar{\alpha})$$ for finite dimensional representations $\alpha,\beta$.

Every representation of a compact quantum group is a direct sum of irreducible representations. For two finite dimensional representations $\alpha$ and $\beta$, denote the number of copies of $\gamma\in\widehat{G}$ in the  decomposition of $\alpha\beta$ into sum of irreducible representations by $N_{\alpha,\beta}^\gamma$. Hence
$$\alpha\beta=\sum_{\gamma\in \widehat{G}}N_{\alpha,\beta}^\gamma\gamma.$$

We have the Frobenius reciprocity law~\cite[Proposition 3.4.]{Woronowicz1987}~\cite[Example 2.3]{Kyed2008}.
$$N_{\alpha,\beta}^\gamma=N_{\gamma,\bar{\beta}}^\alpha=N_{\bar{\alpha},\gamma}^\beta,$$ for all $\alpha,\beta,\gamma\in \widehat{G}$.

Within the paper, we assume that $A=C(G)$ is a separable $C^*$-algebra, which amounts to say, $\widehat{G}$ is countable.

\medskip

\begin{definition}~\cite[Definition 3.2]{Kyed2008}~\label{boundary}\
Given two finite subsets $S, F$ of $\widehat{G}$, the {\bf boundary} of $F$ relative to $S$, denoted by $\partial_S(F)$, is defined by
\begin{align*}
\partial_S(F)=&\{\alpha\in F\,|\, \exists\,\gamma\in S,\, \beta\notin F,\,{\rm such\, that}\, N_{\alpha,\gamma}^\beta>0\,\}   \\
&\cup\{\alpha\notin F\,|\, \exists\,\gamma\in S, \,\beta\in F,\,{\rm such\, that}\, N_{\alpha,\gamma}^\beta>0\,\}.
\end{align*}
\end{definition}

The {\bf weighted cardinality} $|F|_w$ of a finite subset $F$ of $\widehat{G}$ is given by
$$|F|_w=\sum_{\alpha\in F} d_\alpha^2.$$

D. Keyed proves a compact quantum group $G$ is coamenable iff there exists a F{\o}lner sequence in $\widehat{G}$.

\medskip

\begin{theorem}~[F{\o}lner condition for amenable discrete quantum groups]~\label{Fcondition}~\cite[Corollary 4.10]{Kyed2008}\

A compact quantum group $G$ is coamenable iff there exists a sequence $\{F_n\}_{n=1}^\infty$~(a right F{\o}lner sequence) of finite subsets of $\widehat{G}$ such that
$$\lim_{n\to\infty} \dfrac{|\partial_S(F_n)|_w}{|F_n|_w}=0$$ for every finite nonempty subset $S$ of $\widehat{G}$.
\end{theorem}

\section{Mean ergodic theorem for amenable discrete quantum groups}\

In this section we prove the generalized mean ergodic theorem.

\medskip

\begin{theorem}~\label{MeanErgodic}
Let $G$ be a coamenable compact quantum group  with counit $\varepsilon$ and  $\{F_n\}_{n=1}^\infty$ be a right F{\o}lner sequence of $\widehat{G}$. For a representation $\pi:A=C(G)\to B(H)$, we have
\begin{equation}~\label{EqCon}
\lim_{n\to\infty}\frac{1}{|F_n|_w}\sum_{\alpha\in F_n}d_\alpha\pi(\chi(\alpha))=P
\end{equation}
under strong operator topology, where $P$ is the orthogonal projection from $H$ onto $\Hinv=\{x\in H\,|\,\pi(a)x=\varepsilon(a) x\, {\rm for\, all}\, a\in A\}$.
\end{theorem}

We divide the proof into two major steps.

Step 1. We show that $\Hinv=K$ for $K=\{x\in H\,|\,\pi(\chi(\alpha))x=d_\alpha x\, {\rm for\, all}\, \alpha\in \widehat{G}\}.$

Step 2. The sequence $\{\frac{1}{|F_n|_w}\sum_{\alpha\in F_n}d_\alpha\pi(\chi(\alpha))\}_{n=1}^\infty$ converges to the projection from $H$ onto $K$.

\begin{proof}~[Proof of Step 1 for Theorem~\ref{MeanErgodic}]\

\medskip

\begin{lemma}~\label{counit}
If a state $\fai$ on $A=C(G)$ for a compact quantum group $G$ satisfies that $\fai(\chi(\alpha))=d_\alpha$  for all $\alpha\in\widehat{G}$, then $\fai=\varepsilon$.
\end{lemma}
\begin{proof}
It suffices to show that $\fai(u^\alpha_{ij})=\delta_{ij}$ for every $\alpha\in\widehat{G}$ and an arbitrary unitary $U=(u^\alpha_{ij})_{1\leq i,j\leq d_\alpha}\in\alpha$.

Let $\fai(U)$ be the matrix $(\fai(u^\alpha_{ij}))$ in $M_{d_\alpha}(\mathbb{C})$. Note that $\fai$ is  a state, hence completely positive. By a generalized Schwarz inequality of M. Choi~\cite[Corollary 2.8]{Choi1974}, we have
$$\fai(U)\fai(U^*)\leq \fai(UU^*)=1.$$ Let $\Tr$ be the normalized trace of $M_{d_\alpha}(\mathbb{C})$. Since $\fai(\chi(\alpha))=d_\alpha$, we get $\Tr(\fai(U))=1$. It follows that
\begin{align*}
&0\leq \Tr((\fai(U)-1)(\fai(U)-1)^*)=\Tr(\fai(U)\fai(U)^*-\fai(U)^*-\fai(U)+1)  \\
&=\Tr(\fai(U)\fai(U)^*)-1=\Tr(\fai(U)\fai(U^*))-1\leq \Tr(\fai(UU^*))-1=0.
\end{align*}
Hence $\Tr((\fai(U)-1)(\fai(U)-1)^*)=0$ which implies that $\fai(U)=1$. This ends the proof.
\end{proof}

\medskip

\begin{lemma}~\label{Hiv}
Let $\pi:A=C(G)\to B(H)$ be a representation. Then $\Hinv=K=\{x\in H\,|\,\pi(\chi(\alpha))x=d_\alpha x\, {\rm for\, all}\, \alpha\in \widehat{G}\}.$
\end{lemma}
\begin{proof}
Note that $\varepsilon(\chi(\alpha))=d_\alpha$ for all $\alpha\in\widehat{G}$~\cite[Formula (5.11)]{Woronowicz1998}. Hence $\Hinv\subseteq K$.

To show $K\subseteq \Hinv$, we can assume $K\neq 0$ without loss of generality.

Let $x\in K$ be an arbitrarily chosen unit vector. By Lemma~\ref{counit}, the state $\fai_x$ defined by $\fai_x(a)=\langle\pi(a)x,x\rangle$ for all $a\in A$ is $\varepsilon$ since $\fai_x(\chi(\alpha))=d_\alpha$ for all $\alpha\in\widehat{G}$.

For every $a\in A$, we have
\begin{align*}
&\|\pi(a)x-\varepsilon(a)x\|^2=\langle \pi(a)x-\varepsilon(a)x, \pi(a)x-\varepsilon(a)x \rangle \\
&=\langle \pi(a)x, \pi(a)x \rangle-\langle\varepsilon(a)x, \pi(a)x\rangle-\langle \pi(a)x, \varepsilon(a)x \rangle+\langle \varepsilon(a)x,\varepsilon(a)x \rangle  \\
&=\langle \pi(a^*a)x, x \rangle-\langle\varepsilon(a)\pi(a^*)x, x\rangle-\overline{\varepsilon(a)}\langle \pi(a)x, x \rangle+|\varepsilon(a)|^2 \\
&=\varepsilon(a^*a)-\varepsilon(a)\varepsilon(a^*)-|\varepsilon(a)|^2+|\varepsilon(a)|^2=0.
\end{align*}
This proves that $K\subseteq \Hinv$.
\end{proof}
This finishes  proof of Step 1.
\end{proof}

\begin{proof}~[Proof of Step 2 for Theorem~\ref{MeanErgodic}]\

\medskip

\begin{lemma}~\label{purp}
The orthogonal complement $\Hinv^\bot$ of $\Hinv$ is
$$V:=\overline{{\rm Span}\{\pi(\chi(\alpha))x-d_\alpha x\,|\,{\rm for\, all}\, \alpha\in \widehat{G},\,x\in H\}}.$$
\end{lemma}

We need the following well-known fact in functional analysis.

\medskip

\begin{proposition}~\label{prop:purp}
Suppose $\{T_j\}_J$ is a family of bounded operators on a Hilbert space $H$. Then the orthogonal complement of $\bigcap_{j\in J}\ker{T_j}$ is $\overline{{\rm ran}\{T_j^*|j\in J\}}$, the closed linear span of the ranges ${\rm ran}T_j^*$ of $T_j^*$ for all $j$ in $J$.
\end{proposition}

\begin{proof}[Proof of Lemma~\ref{purp}]\

 Consider the family of operators $\{\pi(\chi(\alpha))-d_\alpha\}_{\alpha\in\widehat{G}}$ in $B(H)$. Note that $(\pi(\chi(\alpha))-d_\alpha)^*=\pi(\chi(\bar{\alpha}))-d_{\bar{\alpha}}$, so $\{\pi(\chi(\alpha)-d_\alpha\}_{\alpha\in\widehat{G}}$ is self-adjoint. Applying Proposition~\ref{prop:purp} to $\{\pi(\chi(\alpha)-d_\alpha\}_{\alpha\in\widehat{G}}$ gives the proof.
\end{proof}
Now we are ready  to prove Theorem~\ref{MeanErgodic}.

For every $x\in \Hinv$ and all $n$, we have
$$\frac{1}{|F_n|_w}\sum_{\alpha\in F_n}d_\alpha\pi(\chi(\alpha))x=\frac{1}{|F_n|_w}\sum_{\alpha\in F_n}d_\alpha^2x=x.$$

Next we show that $$\frac{1}{|F_n|_w}\sum_{\alpha\in F_n}d_\alpha\pi(\chi(\alpha))z\to 0$$ for all $z\in V$ as $n\to\infty$. By Lemma~\ref{purp}, we only need to prove it for $z$ of the form $\pi(\chi(\gamma))y-d_\gamma y$ for every $y\in H$ and $\gamma\in\widehat{G}$.

For every $y\in H$ and $\gamma\in \widehat{G}$, we have
\begin{align*}
&\lim_{n\to\infty} \frac{1}{|F_n|_w}\sum_{\alpha\in F_n}d_\alpha\pi(\chi(\alpha))(\pi(\chi(\gamma))y-d_\gamma y)   \\
&=\lim_{n\to\infty} \frac{1}{|F_n|_w}(\sum_{\alpha\in F_n\setminus\partial_\gamma F_n}+\sum_{\alpha\in F_n\cap\partial_\gamma F_n})d_\alpha\pi(\chi(\alpha)\chi(\gamma))y-d_\alpha d_\gamma \pi(\chi(\alpha))y \\
\tag{Theorem~\ref{Fcondition}\quad and \,$\chi(\alpha)\chi(\gamma)=\chi(\alpha\gamma)$}
&=\lim_{n\to\infty} \frac{1}{|F_n|_w}\sum_{\alpha\in F_n\setminus\partial_\gamma F_n}d_\alpha\pi(\chi(\alpha\gamma))y-d_\alpha d_\gamma\pi(\chi(\alpha)) y  \\
\tag{$\alpha\gamma=\sum_{\beta\in F_n}N_{\alpha,\gamma}^\beta\beta$ \,when \, $\alpha\in F_n\setminus\partial_\gamma F_n$}
&=\lim_{n\to\infty} \frac{1}{|F_n|_w}(\sum_{\alpha\in F_n\setminus\partial_\gamma F_n}\sum_{\beta\in F_n}d_\alpha N_{\alpha,\gamma}^\beta\pi(\chi(\beta))y-\sum_{\alpha\in F_n\setminus\partial_\gamma F_n}d_\alpha d_\gamma \pi(\chi(\alpha))y)  \\
\tag{$N_{\alpha,\gamma}^\beta=N_{\beta,\bar{\gamma}}^\alpha$\, and $d_\gamma=d_{\bar{\gamma}}$}
&=\lim_{n\to\infty} \frac{1}{|F_n|_w}(\sum_{\alpha\in F_n\setminus\partial_\gamma F_n}\sum_{\beta\in F_n}d_\alpha N_{\beta,\bar{\gamma}}^\alpha\pi(\chi(\beta))y-\sum_{\alpha\in F_n\setminus\partial_\gamma F_n}d_\alpha d_{\bar{\gamma}} \pi(\chi(\alpha))y) \\
&=\lim_{n\to\infty} \frac{1}{|F_n|_w}(\sum_{\alpha\in F_n\setminus\partial_\gamma F_n}\sum_{\beta\in F_n}d_\alpha N_{\beta,\bar{\gamma}}^\alpha\pi(\chi(\beta))y-\sum_{\alpha\in F_n\setminus\partial_\gamma F_n}[\sum_{\beta\in F_n}+\sum_{\beta\notin F_n}] N_{\alpha,\bar{\gamma}}^\beta d_\beta \pi(\chi(\alpha))y)  \\
\tag{Exchange $\alpha$ and $\beta$ in the second term.}
&=\lim_{n\to\infty} \frac{1}{|F_n|_w}(\sum_{\alpha\in F_n\setminus\partial_\gamma F_n}\sum_{\beta\in F_n}d_\alpha N_{\beta,\bar{\gamma}}^\alpha\pi(\chi(\beta))y-\sum_{\beta\in F_n\setminus\partial_\gamma F_n}[\sum_{\alpha\in F_n}+\sum_{\alpha\notin F_n}] N_{\beta,\bar{\gamma}}^\alpha d_\alpha \pi(\chi(\beta))y)  \\
\tag{Common terms are cancelled.}
&=\lim_{n\to\infty} \frac{1}{|F_n|_w}(\sum_{\alpha\in F_n\setminus\partial_\gamma F_n}\sum_{\beta\in F_n\cap\partial_{\gamma} F_n}d_\alpha N_{\beta,\bar{\gamma}}^\alpha\pi(\chi(\beta))y   \\
&-\sum_{\beta\in F_n\setminus\partial_\gamma F_n}\sum_{\alpha\in F_n\cap\partial_{\gamma} F_n} N_{\beta,\bar{\gamma}}^\alpha d_\beta \pi(\chi(\beta))y-\sum_{\beta\in F_n\setminus\partial_\gamma F_n}\sum_{\alpha\notin F_n} N_{\beta,\bar{\gamma}}^\alpha d_\alpha \pi(\chi(\beta))y)  \\
&=0.
\end{align*}
Note that the last equality above holds since by Theorem~\ref{Fcondition}, we have the following.
\begin{enumerate}
\item
\begin{align*}
& \frac{1}{|F_n|_w}\|\sum_{\alpha\in F_n\setminus\partial_\gamma F_n}\sum_{\beta\in F_n\cap\partial_{\gamma} F_n}d_\alpha N_{\beta,\bar{\gamma}}^\alpha\pi(\chi(\beta))y\|   \\
&\leq  \frac{1}{|F_n|_w}\sum_{\beta\in F_n\cap\partial_{\gamma} F_n}\sum_{\alpha\in F_n}d_\alpha N_{\beta,\bar{\gamma}}^\alpha d_\beta \|y\| \\
&\leq  \frac{1}{|F_n|_w}\sum_{\beta\in F_n\cap\partial_{\gamma} F_n} d_\beta^2 d_{\bar{\gamma}}\|y\|\to 0;
\end{align*}
\item
\begin{align*}
& \frac{1}{|F_n|_w}\|\sum_{\beta\in F_n\setminus\partial_\gamma F_n}\sum_{\alpha\in F_n\cap\partial_{\gamma} F_n}N_{\beta,\bar{\gamma}}^\alpha d_\alpha \pi(\chi(\beta))y\|   \\
&\leq\frac{1}{|F_n|_w}\sum_{\beta\in F_n\setminus\partial_\gamma F_n}\sum_{\alpha\in F_n\cap\partial_{\gamma} F_n}N_{\beta,\bar{\gamma}}^\alpha d_\alpha  d_\beta \|y\| \\
&=\frac{1}{|F_n|_w}\sum_{\beta\in F_n\setminus\partial_\gamma F_n}\sum_{\alpha\in F_n\cap\partial_{\gamma} F_n}N_{\alpha,\gamma}^\beta d_\alpha  d_\beta \|y\| \\
&\leq \frac{1}{|F_n|_w}\sum_{\alpha\in F_n\cap\partial_{\gamma} F_n} d_\alpha^2d_\gamma \|y\|\to 0;
\end{align*}
\item
\begin{align*}
& \frac{1}{|F_n|_w}\|\sum_{\beta\in F_n\setminus\partial_\gamma F_n}\sum_{\alpha\notin F_n} N_{\beta,\bar{\gamma}}^\alpha d_\alpha \pi(\chi(\beta))y\|   \\
&\leq\frac{1}{|F_n|_w}\sum_{\beta\in F_n\setminus\partial_\gamma F_n}\sum_{\alpha\notin F_n} N_{\beta,\bar{\gamma}}^\alpha d_\alpha d_\beta \|y\|\\
&=\frac{1}{|F_n|_w}\sum_{\beta\in F_n\setminus\partial_\gamma F_n}\sum_{\alpha\notin F_n,\,N_{\beta,\bar{\gamma}}^\alpha>0} N_{\beta,\bar{\gamma}}^\alpha d_\alpha d_\beta \|y\|   \\
&\leq \frac{1}{|F_n|_w}\sum_{\beta\in \partial_{\bar{\gamma}} F_n}\sum_{\alpha\in\widehat{G}} N_{\beta,\bar{\gamma}}^\alpha d_\alpha d_\beta  \|y\|\\
&=\frac{1}{|F_n|_w}\sum_{\beta\in \partial_{\bar{\gamma}} F_n} d_\beta^2d_{\bar{\gamma}}\|y\|\to 0
\end{align*}
as $n\to\infty$.
\end{enumerate}
This completes  proof of Step 2.
\end{proof}

For a representation $\pi:B\to B(H)$ of a unital $C^*$-algebra $B$, define the {\bf commutant} $\pi(B)'$ of $\pi(B)$ by
$$\pi(B)'=\{T\in B(H)\,|T\pi(b)=\pi(b)T\,  {\rm for\, all}\, b\in B\}.$$

\medskip

\begin{corollary}~\label{central}
In the setting of Theorem~\ref{MeanErgodic}, the projection $P$ is in $\pi(A)'\cap\overline{\pi(A)}^{SOT}$.
\end{corollary}
\begin{proof}
The left hand side of Equation~\ref{EqCon} is in $\overline{\pi(A)}^{SOT}$, so is $P$. Moreover for all $x,y\in H$ and $a\in A$, we have
$$\langle \pi(a)Px,y\rangle=\varepsilon(a)\langle Px,y\rangle$$ and
\begin{align*}
&\langle P\pi(a)x,y\rangle=\langle \pi(a)x,Py\rangle=\langle x,\pi(a^*)Py\rangle \\
&=\langle x,\pi(a^*)Py\rangle=\langle x,\varepsilon(a^*)Py\rangle=\varepsilon(a)\langle Px,y\rangle.
\end{align*}
This proves $P\in \pi(A)'$.
\end{proof}

As a consequence, we have the following.

\medskip

\begin{corollary}~\label{PureStates}
Assume that $\fai$ is a pure state on $A=C(G)$ for a coamenable compact quantum group $G$ and $\{F_n\}_{n=1}^\infty$ is a right F{\o}lner sequence of $\widehat{G}$. Then
\[ \lim_{n\to\infty}\frac{1}{|F_n|_w}\sum_{\alpha\in F_n}d_\alpha\fai(\chi(\alpha))= \left\{
  \begin{array}{l l}
    1 & \quad \text{if $\fai=\varepsilon,$}  \\
    0 & \quad \text{if $\fai\neq\varepsilon.$}

  \end{array} \right.\]
\end{corollary}

\begin{proof}
When $\fai=\varepsilon$, we have $\varepsilon(\chi(\alpha))=d_\alpha$ for all $\alpha\in\widehat{G}$~\cite[Formula (5.11)]{Woronowicz1998}. Hence
$$\lim_{n\to\infty}\frac{1}{|F_n|_w}\sum_{\alpha\in F_n}d_\alpha\varepsilon(\chi(\alpha))=1.$$

Suppose $\fai\neq\varepsilon$.

Consider the GNS representation $\pi_\fai: A\to B(L^2(A,\fai))$. We have
\begin{align*}
&\lim_{n\to\infty}\frac{1}{|F_n|_w}\sum_{\alpha\in F_n}d_\alpha\fai(\chi(\alpha))  \\
&=\lim_{n\to\infty}\frac{1}{|F_n|_w}\sum_{\alpha\in F_n}d_\alpha\langle \pi_\fai(\chi(\alpha))(\hat{1}),\hat{1}\rangle  \\
&=\langle P(\hat{1}),\hat{1}\rangle.
\end{align*}
Hence $\lim_{n\to\infty}\frac{1}{|F_n|_w}\sum_{\alpha\in F_n}d_\alpha\fai(\chi(\alpha))\neq 0$ iff $P(\hat{1})\neq 0$.

To prove $\lim_{n\to\infty}\frac{1}{|F_n|_w}\sum_{\alpha\in F_n}d_\alpha\fai(\chi(\alpha))=0$ for $\fai\neq\varepsilon$, it suffices to prove that $P(\hat{1})=0$.

Suppose  $P(\hat{1})\neq0$. Then  $\Hinv\neq 0$. By Corollary~\ref{central}, the space $\Hinv$ is an invariant subspace of $L^2(A,\fai)$. Note that $\pi_\fai$ is irreducible since $\fai$ is a pure state. Hence
$\Hinv=L^2(A,\fai)$. In particular $\hat{1}\in \Hinv$. Thus for all $a\in A$, we have $\pi_\fai(a)(\hat{1})=\varepsilon(a)\hat{1}$. It follows that
$$\fai(a)=\langle\pi_\fai(a)(\hat{1}),\hat{1}\rangle=\langle\varepsilon(a)\hat{1},\hat{1}\rangle=\varepsilon(a)$$ for all $a\in A$, which contradicts that $\fai\neq\varepsilon$.
\end{proof}

\section{A Wiener type theorem for compact metrizable groups}\

In this section, we prove the following Wiener type theorem.

\medskip

\begin{theorem}~\label{WienerType}
Let $G$ be a compact metrizable group. Given a $y$ in $G$ and a right F{\o}lner sequence $\{F_n\}_{n=1}^\infty$ of $\widehat{G}$, for a finite Borel measure $\mu$ on $G$, we have
$$\lim_{n\to\infty}\frac{1}{|F_n|_w}\sum_{\alpha\in F_n}d_\alpha\sum_{1\leq i,j\leq d_\alpha}\mu(u^\alpha_{ij})\overline{u^\alpha_{ij}(y)}=\mu\{y\},$$
and $$\lim_{n\to\infty}\frac{1}{|F_n|_w}\sum_{\alpha\in F_n}d_\alpha\sum_{1\leq i,j\leq d_\alpha}|\mu(u^\alpha_{ij})|^2=\sum_{x\in G} \mu\{x\}^2.$$
Hence $\mu$ is continuous iff
$$\lim_{n\to\infty}\frac{1}{|F_n|_w}\sum_{\alpha\in F_n}d_\alpha\sum_{1\leq i,j\leq d_\alpha}|\mu(u^\alpha_{ij})|^2=0.$$ Here $(u^\alpha_{ij})_{1\leq i,j\leq d_\alpha}\in M_{d_\alpha}(C(G))$ stands for a unitary matrix presenting $\alpha\in\widehat{G}$.
\end{theorem}

From now on $G$ stands for a compact metrizable group. When thinking $G$ as a compact quantum group, the coproduct $$\Delta: C(G)\to C(G)\otimes C(G)$$ is given by
$\Delta(f)(x,y)=f(xy)$, the coinverse $\kappa: C(G)\to C(G)$ is given by $\kappa(f)(x)=f(x^{-1})$ and the counit $\varepsilon: C(G)\to\mathbb{C}$ is given by $\varepsilon(f)=f(e_G)$ for all $f\in C(G)$ and $x,y\in G$. Here $e_G$ is the neutral element of $G$.

\medskip

\begin{definition}~\label{ConjugateMeasure}
Given a finite Borel measure $\mu$ on $G$, the {\bf conjugate} $\bar{\mu}$ of $\mu$ is defined by
$$\bar{\mu}(f)=\int_G f(x^{-1})\,d\mu(x)=\mu(\kappa(f))$$ for all $f\in C(G)$, and $\bar{\mu}$ is also a finite Borel measure on $G$. In another word, $\bar{\mu}(E)=\mu(E^{-1})$ for every Borel subset $E$ of $G$.
\end{definition}
For an $x\in G$, use $\delta_x$ to denote the Dirac measure at $x$.

The {\bf convolution} $\mu*\nu$ of two finite Borel measures $\mu$ and $\nu$ on $G$ is defined by
$$\mu*\nu(f)=(\mu\otimes\nu)\Delta(f)=\int_G\int_G \,f(xy)\,d\mu(x)d\nu(y)$$ for all $f\in C(G)$. For every Borel subset $E$ of $G$, we have
$$\mu*\nu(E)=\int_G \nu(x^{-1}E)\,d\mu(x)=\int_G \mu(Ey^{-1})\,d\nu(y).$$ If either $\mu$ or $\nu$ is continuous, then so is $\mu*\nu$.

We can write a finite Borel measure $\mu$ on $G$ as $\mu=\sum_{i}\lambda_i\delta_{x_i}+\mu_C$  for every atom $x_i$ with $\mu\{x_i\}=\lambda_i$ and a finite continuous Borel measure $\mu_C$.

\medskip

\begin{lemma}~\label{MeasureatIdentity}
Let $\mu$ be a finite Borel measure on $G$ and $\{F_n\}_{n=1}^\infty$ be a right F{\o}lner sequence of $\widehat{G}$. Then
$$\lim_{n\to\infty}\frac{1}{|F_n|_w}\sum_{\alpha\in F_n}d_\alpha\mu(\chi(\alpha))=\mu\{e_G\}.$$
\end{lemma}
\begin{proof}
By Corollary~\ref{PureStates}, the sequence $\{\frac{1}{|F_n|_w}\sum_{\alpha\in F_n}d_\alpha\chi(\alpha)(x)\}\subseteq C(G)$ pointwisely converges to $1_{e_G}$~(the characteristic function of $\{e_G\}$). Note that $|\frac{1}{|F_n|_w}\sum_{\alpha\in F_n}d_\alpha\chi(\alpha)(x)|\leq 1$ for all $x\in G$, hence by Legesgue's Dominated Convergence Theorem~\cite[1.34]{Rudin1987}, we have
\begin{align*}
&\lim_{n\to\infty}\frac{1}{|F_n|_w}\sum_{\alpha\in F_n}d_\alpha\mu(\chi(\alpha))=\lim_{n\to\infty}\int_G \frac{1}{|F_n|_w}\sum_{\alpha\in F_n}d_\alpha\chi(\alpha)(x)\,d\mu(x) \\
&=\int_G\lim_{n\to\infty} \frac{1}{|F_n|_w}\sum_{\alpha\in F_n}d_\alpha\chi(\alpha)(x)\,d\mu(x)=\int_G 1_{e_G}\,d\mu=\mu\{e_G\}.
\end{align*}
\end{proof}

\begin{proof}~[Proof of Theorem~\ref{WienerType}]\

Given a finite Borel measure $\mu$ on $G$ and a $y\in G$, consider the measure $\mu*\delta_{y^{-1}}$. By Lemma~\ref{MeasureatIdentity}, we have
$$\lim_{n\to\infty}\frac{1}{|F_n|_w}\sum_{\alpha\in F_n}d_\alpha\mu*\delta_{y^{-1}}(\chi(\alpha))=\mu*\delta_{y^{-1}}\{e_G\}.$$ Note that
\begin{align*}
&\mu*\delta_{y^{-1}}(\chi(\alpha))=\int_G\int_G \chi(\alpha)(xz)\,d\mu(x)d\delta_{y^{-1}}(z)   \\
&=\int_G \chi(\alpha)(xy^{-1})\,d\mu(x)=\int_G \sum_{1\leq i\leq d_\alpha} u^\alpha_{ii}(xy^{-1})\,d\mu(x)   \\
&=\int_G \sum_{1\leq i\leq d_\alpha}\sum_{1\leq j\leq d_\alpha} u^\alpha_{ij}(x)u^\alpha_{ji}(y^{-1})\,d\mu(x)  \\
&=\int_G \sum_{1\leq i\leq d_\alpha}\sum_{1\leq j\leq d_\alpha} u^\alpha_{ij}(x)\overline{u^\alpha_{ij}(y)}\,d\mu(x).
\end{align*}
Moreover $$\mu*\delta_{y^{-1}}\{e_G\}=\int_G\int_G 1_{e_G}(xz)\,d\mu(x)d\delta_{y^{-1}}(z)=\int_G 1_{e_G}(xy^{-1})\,d\mu(x)=\mu\{y\}.$$ This completes the proof of part 1.

Applying Lemma~\ref{MeasureatIdentity} to $\mu*\bar{\mu}$, we have
$$\lim_{n\to\infty}\frac{1}{|F_n|_w}\sum_{\alpha\in F_n}d_\alpha\mu*\bar{\mu}(\chi(\alpha))=\mu*\bar{\mu}\{e_G\}.$$

Since $\mu=\sum_{x_i \,{\rm atoms}}\lambda_i\delta_{x_i}+\mu_C$  with  $\lambda_i=\mu\{x_i\}$ and $\mu_C$ a finite continuous Borel measure, we have
$$\bar{\mu}=\sum_{x_i \,{\rm atoms}}\lambda_i\overline{\delta_{x_i}}+\overline{\mu_C}=\sum_{x_i \,{\rm atoms}}\lambda_i\delta_{x_i^{-1}}+\overline{\mu_C}.$$

Hence
\begin{align*}
\mu*\bar{\mu}=&\sum_i\sum_j\lambda_i\lambda_j \delta_{x_i}*\delta_{x_j^{-1}}   \\
&+\sum_i \lambda_i\delta_{x_i}*\overline{\mu_C}+\sum_j\lambda_j\mu_C*\delta_{x_j^{-1}}+\mu_C*\overline{\mu_C}.
\end{align*}

Note that $\sum_i \lambda_i\delta_{x_i}*\overline{\mu_C}+\sum_j\lambda_j\mu_C*\delta_{x_j^{-1}}+\mu_C*\overline{\mu_C}$ is a finite continuous measure and
$\sum_{i,j}\lambda_i\lambda_j \delta_{x_i}*\delta_{x_j^{-1}}=\sum_{i,j}\lambda_i\lambda_j \delta_{x_ix_j^{-1}}.$  It follows that
$$\mu*\bar{\mu}\{e_G\}=\sum_{x_i \,{\rm atoms}} \lambda_i^2=\sum_{x_i \,{\rm atoms}} \mu\{x_i\}^2=\sum_{x\in G} \mu\{x\}^2.$$

On the other hand,
\begin{align*}
&\mu*\bar{\mu}(\chi(\alpha))=\int_G\int_G \chi(\alpha)(xy)\,d\mu(x)d\bar{\mu}(y)   \\
&=\int_G \int_G\chi(\alpha)(xy^{-1})\,d\mu(x)d\mu(y)=\int_G \int_G\sum_{1\leq i\leq d_\alpha} u^\alpha_{ii}(xy^{-1})\,d\mu(x)d\mu(y)   \\
&=\int_G \int_G\sum_{1\leq i\leq d_\alpha}\sum_{1\leq j\leq d_\alpha} u^\alpha_{ij}(x)u^\alpha_{ji}(y^{-1})\,d\mu(x)d\mu(y)  \\
&=\sum_{1\leq i\leq d_\alpha}\sum_{1\leq j\leq d_\alpha}\int_G  u^\alpha_{ij}(x)\,d\mu(x)\int_G\overline{u^\alpha_{ij}(y)}\,d\mu(y) \\
&=\sum_{1\leq i,j\leq d_\alpha}|\mu(u^\alpha_{ij})|^2.
\end{align*}
This ends the proof of part 2, and part 3 follows immediately.
\end{proof}

\end{document}